\documentclass[11pt]{amsart}

\usepackage{graphicx}
\usepackage{amsmath}
\usepackage{amssymb}
\usepackage{mathrsfs} 
\usepackage{xcolor}
\usepackage{mathabx}
\usepackage{caption}
\usepackage{float}
\usepackage{enumerate}

\usepackage[utf8]{inputenc}

\newtheorem{thm}{Theorem}[section]

\newtheorem{lem}[thm]{Lemma}
\newtheorem{cor}[thm]{Corollary}

\newtheorem{rem}[thm]{Remark}

\theoremstyle{definition}
\newtheorem{definition}[thm]{Definition}

\theoremstyle{remark}

\newtheorem{remark}[thm]{Remark}

\numberwithin{equation}{section}

\newcommand{\Schw}{\mathcal{S}}
\newcommand{\R}{\mathbb{R}}  
\newcommand{\C}{\mathbb{C}} 
\newcommand{\A}{\mathbb{A}} 
\newcommand{\Ab}{\mathcal{A}} 
\newcommand{\Z}{\mathbb{Z}}
\newcommand{\G}{\mathbb{G}}

\newcommand{\T}{\mathbb{T}}

\newcommand{\Lam}{\Lambda_{*}}
\newcommand{\Leb}{\mathcal{L}^{n}}

\newcommand{\B}{\mathbb{B}}

\newcommand{\HH}{\mathcal{H}}
\newcommand{\dH}{\dim_{\mathcal{H}}}
\DeclareMathOperator{\M}{M}

\DeclareMathOperator{\supp}{supp}

\DeclareMathOperator{\spann}{span}
\DeclareMathOperator{\Mat}{Mat}

\newcommand{\Lamb}{\Lambda_{*}^{\mathbb{B}}}

\DeclareMathOperator{\spec}{spec}
\DeclareMathOperator{\Tan}{Tan}

\DeclareMathOperator{\dist}{dist}

\begin{document}

\title[On finite configurations in the spectra of singular measures]{On finite configurations in the spectra of singular measures}

\subjclass[2020]{11B25, 28A78, 42B10}
\keywords{uncertainty principle, finite configurations, Hausdorff dimension, PDE-constrained measures, multifractal analysis}

\author{R. Ayoush}
\address{Leonard Euler International Mathematical Institute in St. Petersburg, Pesochnaya nab. 10, St. Petersburg, 197022, Russia}

\address{Institute of Mathematics, University of Warsaw, ul. Banacha 2, 02097, Warsaw, Poland}

\email{rami.m.ayoush@gmail.com}

\begin{abstract}
We establish various forms of the following certainty principle: a set $S \subset \R^{n}$ contains a given finite linear pattern, provided that $S$ is a support of the Fourier transform of a sufficiently singular probability measure on $\R^{n}$. As its main corollary, we provide new dimensional estimates for PDE- and Fourier-constrained vector measures. Those results, in certain cases of restrictions given by homogeneous operators, improve known bounds related to the notion of the $k$-wave cone.
\end{abstract}

\maketitle

\section{Introduction}

In this article we study the problem of the Hausdorff dimension estimates from the point of view of elementary additive combinatorics. We show that in order to obtain lower bounds for the dimension of a probability measure, a good strategy may be to examine structural properties of its Fourier spectrum, i.e. the support of its Fourier-Stieltjes transform. In our case, this leads to estimates of the density of the so-called finite configurations (see \cite{CLP}) in the spectrum. Let us begin with a precise formulation of this dependence.

\begin{definition}
	Let  $\mu \in \M(\R^{n})$ be a finite, Borel-regular measure. We call the Fourier spectrum of $\mu$ the set
	\[
	\spec(\mu) := \overline{\{\xi \in \R^{n}  \colon \widehat{\mu}(\xi) \neq 0\}},
	\]
	i.e. the support of the Fourier-Stieltjes transform of $\mu$.
\end{definition}

\begin{definition}\label{localdim}
	The local lower and local upper Hausdorff dimension of a measure $\mu \in \M^{+}(\R^{n})$ at a point $x \in \supp{\mu}$ are defined (respectively) by 
	\begin{equation*}
		\underline{D}\mu(x) := \liminf_{r \to 0} \frac{\log (\mu(B(x,r)))}{\log r}
	\end{equation*} 
	and
	\begin{equation*}
		\overline{D}\mu(x) := \limsup_{r \to 0} \frac{\log (\mu(B(x,r)))}{\log r},
	\end{equation*} 
	where $B(x,r)$ denotes the closed ball of radius $r$ centered at $x$.
	If the above quantities are equal, then the common limit is denoted by $D \mu(x)$.
\end{definition}

\begin{definition}\footnote{This definition slightly differs from the one in \cite{CLP}.}
	Let $k\geq 1$ and let $B_{1}, \dots, B_{k} \in \Mat_{n\times n}(\R)$ be a sequence of $n \times n$ matrices with real entries. We denote by $\B$ the $kn\times n$ matrix given by 
	\begin{equation*}
		\B =  \begin{bmatrix}
			B_{1} \\
			B_{2} \\
			\vdots \\
			B_{k}
		\end{bmatrix}.
	\end{equation*}
	For $x\in \R^{n}$, a finite set of points
	\begin{equation*}
		\{x\} \cup \{B_{1}x\}\cup \dots \cup \{B_{k}x\}
	\end{equation*}
	is called a $\B$-configuration generated by $x$.
	We say that a $\B$-configuration is non-trivial if it is generated by a non-zero $x$.
\end{definition}
Our first main result is the following.
\begin{thm}\label{mainres2}
	Let $k\geq 1$ and let $B_{1}, \dots, B_{k} \in \Mat_{n\times n}(\R)$ be a sequence of $n \times n$ matrices. Suppose that $\mu \in \M^{+}(\R^{n})$ is a finite, non-negative measure such that at some $x \in \R^{n}$ we have
	\begin{equation}\label{condition2}
		0 \leq \underline{D} \mu(x) < \frac{n}{k+1}.
	\end{equation}
	Then $\spec(\mu)$ contains a non-trivial $\B$-configuration. Moreover, the set of such $\B$-configurations is large in the sense of the Lebesgue measure $\Leb$, i. e.
	\begin{equation*}
		\Leb(\{\xi \in \spec(\mu) \colon \xi \ \text{generates a } \B\text{-configuration}\})= +\infty.
	\end{equation*}
\end{thm}
In the theorem above one can replace $\underline{D}\mu(x)$ with (non-local) lower Hausdorff dimension of a measure, i.e.
\begin{equation}\label{dimension}
	\dH(\mu) := \inf \{\dH(F) : |\mu|(F) \neq 0 \}.
\end{equation}
Indeed, in view of a  well-known characterization of the lower Hausdorff dimension of a measure (see Proposition 10.2. in \cite{Falconer}, and \cite{F06}), one has:
\begin{equation}
	\dH(\mu) = \sup\{s: \underline{D}\mu(x) \geq s \ \text{for} \ \mu\text{-almost every} \ x\}.
\end{equation}

Theorem \ref{mainres2} allows to improve several dimensional bounds for certain classes of vector measures. Let us briefly discuss known results in this direction. 

In the paper \cite{RW} the problem of dimensional estimates for vector measures was considered under restrictions on the phase of the Fourier transform. More precisely, let $\phi : S^{n-1} \to \G_{\C}(d, E)$ be a continuous function, called henceforth a bundle\footnote{Formally, it is $\{(\xi, \phi(\xi)): \xi \in S^{n-1}\}$ who is a vector bundle over $S^{n-1}$ in the common meaning.}, taking values in the Grassmannian of $d$-dimensional (complex) subspaces of some finite dimensional complex vector space $E$. The class of measures $\M_{\phi}(\R^{n}, E)$, subordinated to $\phi$, is defined as follows.
\begin{definition}[\cite{RW}]
	Let $\phi: S^{n-1} \to \G_{\C}(d, E)$ be a continuous bundle. The set $\M_{\phi}(\R^{n}, E)$ consists of finite vector measures $\mu \in \M(\R^{n}, E)$ such that
	\begin{equation*}
		\forall{\xi \in \R^{n} \setminus \{0\}} \quad \widehat{\mu}(\xi) \in \phi\bigg{(}\frac{\xi}{|\xi|}\bigg{)}.
	\end{equation*}
\end{definition}

The definition of the class $\M_{\phi}(\R^{n}, E)$ was inspired by measures arising as a result of action of homogeneous vector-valued differential operators on a scalar function, i.e.
\begin{equation*}
	\mu = (P_{1}f, P_{2}f, \dots, P_{n}f) \quad \text{for} \ f \in L^{1}(\R^{n}).
\end{equation*}
A particular example of such a measure is the gradient of $f \in BV(\R^{n})$. Results concerning dimensional and singular properties of gradients of functions with bounded variation belong to the classics of Geometric Measure Theory and are well understood (see Chapter 3 in \cite{AFP}). In particular, it is known that any such measure has lower Hausdorff dimension at least $n-1$ (see Lemma 3.76 in \cite{AFP}).

One of the results proved in \cite{RW} (Theorem 3 therein) is a theorem which delivers a criterion when the dimension of such vector measure is at least one. 

\begin{thm}[Theorem 3 in \cite{RW}] Suppose that:
	\begin{enumerate}
		\item $\phi$ is a $\gamma$-H{\"o}lder continuous map with the exponent $\gamma > \frac{1}{2}$,
		\item the bundle $\phi$ satisfies the structural assumption
		\begin{equation}\label{1cone}
			\bigcap_{\xi \in S^{n-1}} \phi(\xi) = \{0\}.
		\end{equation}
	\end{enumerate}
	Then, each finite vector measure $\mu \in \M_{\phi}(\R^{n}, E)$ satisfies the dimensional bound
	\begin{equation*}
		\dH(\mu) \geq 1.
	\end{equation*}
\end{thm}

 In the theorem above, the complex Grassmannian $\G_{\C}(d, E)$ is equipped with the standard metric, equivalent to the angle between subspaces:
\begin{equation}\label{grasmet}
	\dist_{\G_{\C}(d, E)}(V,W) = \sup_{z \in V \cap S^{N-1}} \dist_{E}(z,W),
\end{equation} 
where $N = \dim E$.

  Recent years yielded a growth of interest in properties of measures constrained by differential operators in various ways (cf. \cite{arroyo}, \cite{ADHR}, \cite{DR}, \cite{dima}, \cite{SW}). One of the objects of those studies are $\Ab$-free measures which are defined as follows.
 
 Let $\Ab$ be an $m$-th order linear partial differential operator on $\R^{n}$ with constant coefficients, i.e.

\begin{equation}
 \Ab f = \sum_{|\alpha| \leq m} A_{\alpha} \partial^{\alpha} f
\end{equation}
for any $\R^{l}$-valued distribution $f \in \mathcal{D}'(\R^{n}; \R^{l})$. Here, $A_{\alpha} \in \Mat_{s\times l}(\R)$ are constant $s \times l$ matrices with real entries.
A vector-valued Radon measure $\mu \in \M(\R^{n}; \R^{l})$ is called an $\Ab$-free measure, if is it solves the equation
\begin{equation}
	\Ab \mu = 0
\end{equation}
in the sense of distributions. Our main goal is to prove non-sharpness of known dimension bounds relating the (lower) Hausdorff dimension of $\Ab$-free measures to the hierarchy of $k$-wave cones of $\Ab$ in the case when $\Ab$ is a homogeneous operator, and to provide new estimates which are better for some operators. 
The $k$-wave cone of $\Ab$ (see \cite{ADHR}), here denoted by $\Lambda^{k}(\Ab)$, is a set described by the following relation:
\begin{equation}
	\Lambda^{k} (\Ab) := \bigcap_{V \in \G(k,\R^{m})} \bigcup_{\xi \in V\setminus \{0\}} \ker \ \A^{m}[\xi],
\end{equation}
where
\begin{equation*}
	\A^{m}[\xi] = \sum_{|\alpha| = m} A_{\alpha} \xi^{\alpha}
\end{equation*}
and $\G(k, \R^{m})$ stands for the Grassmannian of (real) $k$-dimensional subspaces of $\R^{m}$.
\begin{thm}[Theorem 1.3., Corollary 1.4. in \cite{ADHR}] Let $\Ab$ be a linear partial differential operator as above with an empty $k$-wave cone, i.e.
\begin{equation*}
	\Lambda^{k}(\Ab) = \{0\}.
\end{equation*}	 
Then, for any $\Ab$-free measure $\mu \in \M(\R^{n}; \R^{l})$  we have
\begin{equation*}
	\dH(\mu) \geq k.
\end{equation*}
\end{thm}

One can observe that if $\Ab$ is a constant rank homogeneous operator and we take $\phi(\xi) = \ker \ \A^{k}[\xi]$ then \eqref{1cone} is equivalent to the condition that the $1$-wave cone of $\Ab$-free measures is empty. This fact justifies the plan of extending estimates known for differential operators, related to the hierarchy of wave cones, to the already described general Fourier setting. This goal, for smooth bundles, was achieved in \cite{dima} (see also \cite{AW} for results corresponding to the modified $2$-wave cones and Lipschitz-regular bundles) by relating certain multifractal properties of tangent measures from $\Tan(\mu, x)$ (in the sense of Preiss, \cite{P}) to the Hausdorff dimension of $\mu$ via Harnack inequalities.

On the other hand, it is also natural to ask whether such bounds are sharp in any of those settings. In the general case of non-homogeneous differential operators, a negative answer is delivered by results from \cite{SW}. Indeed, it is shown there that for the case when $\mu$ is given by differential operator consisting of pure derivatives, i.e.
\begin{equation*}
	\mu = (\partial^{\alpha_{1}}_{1}f, \partial^{\alpha_{2}}_{2}f, \dots, \partial^{\alpha_{n}}_{n}f) \quad \text{for} \ f \in L^{1}(\R^{n}),
\end{equation*}
then we have $\dH(\mu) \geq n-1$. On the other hand, if indices $\alpha_{j}$ are non-equal, then each $k$-wave cone is usually non-empty.
It turns out that in the case of homogeneous operators, the dependence of dimension and the optimal order of empty $k$-wave cone is also non-sharp, which is confirmed by our next result.

\begin{thm}\label{sharper}
	There exists $\Ab$, a constant coefficient homogeneous differential operator on $\R^{3}$ such that $\Lambda^{2}(\Ab) \neq \{0\}$ and
	\begin{equation*}
		\dH (\mu) \geq \frac{3}{2}
	\end{equation*}
	for each $\Ab$-free measure $\mu$.
\end{thm}
 The above is a consequence of a new dimensional bound, which is proved by counting finite configurations inside level sets of bundles, i.e. sets of the form
\begin{equation*}
	\phi^{-1}(v) = \{\xi\in S^{n-1}: v \in \phi(\xi)\},
\end{equation*}
for $v\in E$. In its proof we rely on the strategy and ideas from \cite{dima}. The criterion is obtained by replacement of Harnack inequalities with Theorem \ref{mainres2} in the reasoning from \cite{dima}, and it reads as follows.

\begin{thm}\label{comb2}
Let $\phi: S^{n-1} \to \G_{\C}(d, E)$ be a smooth bundle. Suppose that for each $v \in E$ there exists an orthogonal map $B_{v} \in O(n)$ such that
\begin{equation*}
	\phi^{-1}(v) \cap B_{v}(\phi^{-1}(v)) = \emptyset.
\end{equation*}
Then
\begin{equation*}
	\dH(\mu) \geq \frac{n}{2}
\end{equation*}
for each $\mu \in \M_{\phi}(\R^{n}, E)$.
\end{thm}
Let us briefly describe the structure of the paper. Section \ref{chap_count} contains the proof of Theorem \ref{mainres2} and a discussion of its alternative forms. For example, we provide criteria for existence of non-trivial arithmetic sequences in the Fourier spectrum. We  also describe  there connections of those theorems with results of Łaba and Pramanik (\cite{Laba}) and the theorem of Chan, Łaba and Pramanik (\cite{CLP}).
 In Section \ref{chap_comb} we prove Theorem \ref{comb2} (in a slightly more general form) in a way that we describe necessary modifications of the method from \cite{dima} leading to the proof of this theorem. In the last section, we prove Theorem \ref{sharper} by constructing an explicit example of the claimed differential operator which defines a bundle satisfying assumptions of Theorem \ref{comb2}.
 \subsection{Notation} For $X=\T$ or $X=\R^{n}$ we denote by $\M(X)$ the space of finite, Borel-regular measures on $X$ and by $\M^{+}(X)$ the cone of finite, non-negative measures on $X$. 
 
 By $\lVert \cdot \rVert$ we understand the usual Euclidean norm on $\R^{n}$. The symbol $B(x,r)$ stands for the closed Euclidean ball of radius $r$ centered at $x$. 
 
  The $\alpha$-dimensional Hausdorff measure is denoted by $\mathcal{H}^{\alpha}$ and $\mathcal{H}^{\alpha}_{|S}$ is its restriction to a given set $S$.
 
 We denote by $\G(d, E)$ the Grassmanian of $d$-dimensional real subspaces of a vector space $E$ and by $\G_{\C}(d, E)$ the analogous Grassmanian of complex vector spaces (in the case when $E$ is a complex space). 
 
 For a matrix $A$ we denote by $A^{T}$ its transpose. By $O(n)$ we understand the set of orthogonal transformations on $\R^{n}$.
 
  For two real numbers $A,B$ the relation $A\lesssim B$ means that there exists a positive constant $C$ such that $A \leq C\cdot B$.

 The Fourier-Stieltjes coefficients are computed accordingly to the below formula:
 \[
 \widehat{\mu}(n) = \frac{1}{2\pi} \int^{2\pi}_{0}e^{-int}d\mu(t).
 \]
 
 The Fourier-Stieltjes transform on $\R^{n}$ is normalized as follows:
 
 \begin{equation*}
 	\widehat{\mu}(\xi) = \int_{\R^{n}}e^{-2\pi i \langle x, \xi\rangle} d\mu(x).
 \end{equation*}

\section{Applications of multilinear forms counting finite configurations}\label{chap_count}

\subsection{Proof of Theorem \ref{mainres2}}
We rely on elementary properties of multilinear forms that count $\B$-configurations (cf. Section 2 in \cite{CLP}):
\begin{equation}\label{multform}
	\Lamb(g;f_{1}, f_{2}, \dots, f_{k}) = \int_{\R^{n}}\widehat{g}(\xi) \prod^{k}_{j=1} \widehat{f_{j}}(B_{j}\xi) d\xi.
\end{equation}
\begin{lem}\label{multrepr}
	If $g, f_{1}, \dots, f_{k} \in \Schw(\R^{n})$ then
	\begin{equation}
		\Lamb(g;f_{1}, f_{2}, \dots, f_{k}) = \int_{(\R^{n})^{k}} g(-\B^{T}\vec{x}) \prod_{j=1}^{k} f_{j}(x_{j}) dx_{1}\dots dx_{k},
	\end{equation}
	where $\vec{x} = (x_{1}, \dots, x_{k})^{T} \in \R^{nk}$ is a column vector obtained by concatenation of $x_{1}, \dots, x_{k}$.
\end{lem}

\begin{proof}
By using the definition of the Fourier transform and the Fubini theorem, we obtain
	\begin{multline}
		\Lamb(g;f_{1}, f_{2}, \dots, f_{k}) = \int_{\R^{n}}\widehat{g}(\xi) \prod^{k}_{j=1} \bigg{(} \int_{\R^{n}} e^{-2\pi i \langle x_{j}, B_{j} \xi \rangle} f_{j}(x_{j}) dx_{j} \bigg{)}  d\xi \\
		= \int_{\R^{n}}\widehat{g}(\xi) \prod^{k}_{j=1} \bigg{(} \int_{\R^{n}} e^{-2\pi i \langle B_{j}^{T} x_{j}, \xi \rangle} f_{j}(x_{j}) dx_{j} \bigg{)}  d\xi \\ = \int_{(\R^{n})^{k}} \bigg{(} \prod_{j=1}^{k} f_{j}(x_{j}) \cdot \int_{\R^{n}} e^{-2\pi i \langle \B^{T}\vec{x}, \xi \rangle} \widehat{g}(\xi) d\xi \bigg{)} dx_{1}dx_{2}\dots dx_{k} \\
		= \int_{(\R^{n})^{k}} \bigg{(} \prod_{j=1}^{k} f_{j}(x_{j}) \cdot g(-\B^{T} \vec{x}) \bigg{)} dx_{1}dx_{2}\dots dx_{k},
	\end{multline}
	which proves the lemma.
\end{proof}

\begin{proof}[Proof of Theorem~\ref{mainres2}]
By the translation invariance, we may assume that \eqref{condition2} holds at $x = 0$ and without loss of generality it can be replaced with the assumption that for any $\alpha$ such that $\underline{D} \mu(0) < \alpha <\frac{n}{k+1}$ there exists a sequence $r_{j} \downarrow 0$ such that
	\begin{equation*}
		\mu(B(0,r_{j})) \geq r_{j}^{\alpha}.
	\end{equation*}
Let us denote the standard Gaussian kernel by

\begin{equation*}
	g_{t}(x) = \frac{1}{(2\pi)^{\frac{n}{2}}t^{n}} \exp\bigg{(}\frac{-\lVert x \rVert^{2}}{2t^{2}}\bigg{)} \quad \text{for} ~x\in\R^{n}.
\end{equation*}
We have
\begin{equation*}
	\widehat{g_{t}}(\xi) = e^{-2\pi^{2}t^{2}\lVert \xi \rVert^{2}} \quad \text{for} ~ \xi \in \R^{n}.
\end{equation*}
Moreover, if we pick 
	\begin{equation*}
		G_{t}(x) = g_{t} \ast \mu (x) = \int_{\R^{n}} g_{t}(x-y) d\mu(y),
	\end{equation*}
then
\begin{equation*}
	|\widehat{G_{t}}(\xi)| \leq \lVert g_{t} \ast \mu \rVert_{1} \lesssim 1.
\end{equation*}
From the defining identity \eqref{multform} and the inequality above, for any $t\in(0,+\infty)$ we obtain
	\begin{equation}\label{upperbound}
		|\Lamb(G_{t}; G_{t}, \dots, G_{t})| \lesssim \Leb\big{(}\{ \xi \in \spec(\mu) \colon  B_{i} \xi \in \spec(\mu) ~\forall{i=1,\dots,k}\} \big{)}.
	\end{equation}
	Moreover, for $x \in B(0,r_{j})$ we have
	\begin{multline}\label{gausses}
		G_{r_{j}}(x) = \int_{\R^{n}} g_{r_{j}}(x-y) d\mu(y) \geq \int_{B(0,r_{j})} g_{r_{j}}(x-y) d\mu(y) \gtrsim \frac{1}{r_{j}^{n}}\mu(B(0,r_{j})) \\ \gtrsim r_{j}^{\alpha-n}.
	\end{multline}
	Let $c = \lVert B_{1}^{T} \rVert + \dots + \lVert B_{k}^{T} \rVert$ be the sum of operator norms of matrices $B_{i}^{T}$, $i = 1, \dots, k$. By Lemma \ref{multrepr} and \eqref{gausses} we have
	\begin{equation}
		\Lamb(G_{r_{j}}; G_{r_{j}}, \dots, G_{r_{j}}) \gtrsim \Leb\bigg{(}B(0,r_{j}/c)\bigg{)}^{k} (r_{j}^{\alpha-n})^{k+1} \simeq r_{j}^{\alpha(k+1) - n}.
	\end{equation}
	Finally, by the above and \eqref{upperbound}
	\begin{equation*}
		r_{j}^{\alpha(k+1)-n} \lesssim \Leb\big{(}\{ \xi \in \spec(\mu) \colon  B_{i} \xi \in \spec(\mu) ~\forall{i=1,\dots,k}\} \big{)}.
	\end{equation*}
	Thus, if $\alpha < \frac{n}{k+1}$, then 
	\begin{equation*}
		\Leb\big{(}\{ \xi \in \spec(\mu) \colon  B_{i} \xi \in \spec(\mu) ~\forall{i=1,\dots,k}\} \big{)} = +\infty,
	\end{equation*}
	which proves the theorem.
\end{proof}

\subsection{Other certainty principles and their connections with the Łaba-Pramanik theorem}

Let us begin the discussion of connections of our (rather simple) approach with classical (and much more difficult) results from the case of counting 3-term arithmetic progressions (called further 3AP's). Because of the symmetry of Fourier coefficients of non-negative measures, to make our result meaningful, in their formulations by trivial 3AP's we understand  not only constant sequences, but also those which are symmetric with respect to zero (i.e. those, whose second term is zero). 

A classical result by Roth (see \cite{Roth}) says that any set $S \subset \Z$ with positive upper density must contain a non-trivial arithmetic progression. It is also commonly known, that the existence of 3AP's can be also proved for more sparse sets, in presence of some additional assumptions about concentration and/or randomness. A prominent example of such a result is the theorem of Łaba and Pramanik (Theorem 1.2 in \cite{Laba}). It says, loosely speaking, that if a subset of the circle group $E \subset \T$ supports a probability measure $\nu$ such that Fourier coefficients of $\nu$ decay rapidly, and $\nu$ satisfies the $\alpha$-Frostman condition with $\alpha$ sufficiently close to one, then $E$ must contain a 3AP.  In other words, it expresses a tradeoff between structural properties of sets and quantities which favourize random and evenly distributed measures. In our case, which may be thougt of a simple 'dual' analog of the Łaba-Pramanik theorem (see \cite{Potgi} for another analog for the group $\Z$), we assume $S$ to be the spectrum of a singular measure. Our criterion is the following.

\begin{thm}\label{mainres}
	Suppose that $\mu \in \M^{+}(\T)$ is a finite, non-negative measure such that at some $x \in \T$ we have
	\begin{equation}\label{condition}
		0 \leq \underline{D} \mu(x) < \frac{1}{3}.
	\end{equation}
	Then $\spec(\mu)$ contains a non-trivial 3AP.
\end{thm}

 The hidden asumption about randomness appears here as an instance of the Uncertainty Principle (see \cite{HJ} for many other examples). Indeed, one can read \eqref{condition} as a condition which may witness about both strong concentration and deterministic nature of a measure on $\T$. This condition, by the UP, in a sense gets opposite meaning on the frequency side. 
 
 From the technical point of view, proofs of the above results rely on the fact that suitable trilinear forms which are used for counting 3AP's (and finite configurations in the multidimensional case) have convenient representations both in time and frequency domain.

Theorem \ref{mainres2} is a multidimensional analog of Theorem \ref{mainres}. Its relation to  results by Chan, Łaba and Pramanik from \cite{CLP} is similar to the relation beteween Theorem \ref{mainres} and \cite{Laba}.

Theorem \ref{mainres} can be strengthened to a more quantitative form which takes into account the (upper) fractional density of arithmetic progressions in the Fourier spectrum.

\begin{definition}\label{tdef}
	For any $A\subset \Z$ let $t_{n}(A)$ be the function given by
	\begin{equation}
		t_{n}(A) = \#\{(m,r) \in \Z\times \Z \colon m, m+r, m+2r \in A \cap [-n,n]\}.
	\end{equation}
\end{definition}
\begin{thm}\label{mrquant}
	Suppose that $\mu \in \M^{+}(\T)$ is a finite, non-negative measure such that at some $x \in \T$ we have
	\begin{equation}\label{con2}
		0 \leq \underline{D} \mu(x) < \gamma.
	\end{equation} 
	Then
	\begin{equation}
		\limsup_{n\to \infty}\frac{t_{n}(\spec(\mu))}{n^{\beta}} = +\infty
	\end{equation}
	for $0\leq \beta < 2-3\gamma$.
\end{thm}

\begin{cor}\label{wniosek}
	Let $\mu \in \M^{+}(\T)$ and suppose that 
	\begin{equation}
		t_{n}(\spec(\mu)) \leq Cn^{\beta}
	\end{equation} 
	for some $0\leq \beta\leq 2$. Then
	\begin{equation}
		\dH(\mu) \geq \frac{2-\beta}{3}.
	\end{equation}
\end{cor}

\begin{rem}
Other dimensional estimates, taking into account arithmetical properties of elements of spectrum were proved in \cite{ASW2}: 

Let $q \geq 3$ be a fixed integer and suppose that $B$ is a proper subset of $\{1,\dots,q-1\}$. Then, if $\mu \in \M^{+}(\T)$ is a measure such that $\widehat{\mu}$ is supported on integers of the form $q^{N}m$, where the residue of $m$ modulo $q$ belongs to $B$, then
\begin{equation*}
	\dH(\mu) \geq c(B) > 0,
\end{equation*}
where $c(B)$ is a constant depending on $B$ only.
This result was derived from a theorem proved in \cite{ASW}, concerning dimensional estimates for measures satisfying martingale analogs of constraints given by bundles.
\end{rem}

\subsection{Sketches of proofs of Theorem \ref{mainres} and Theorem \ref{mrquant}}
In this section we give hints to the proofs of the announced theorems. The general approach is similar to the proof of Theorem \ref{mainres2} and it concerns estimating a suitable multilinear form. In this case, we use some modification of a well-known object, which is frequently used tool to attack problems involving 3AP's, the so-called trilinear form. Our modification is the following:
\begin{equation}
	\Lam(f,g,h) := \sum_{n,r \in \Z} \widehat{f}(n) \widehat{g}(n+r) \widehat{h}(n+2r)
\end{equation}
\begin{lem}\label{representation} Let f,g,h be trigonometric polynomials on $\T$. The following identity holds:
	\begin{equation}
		\Lam(f,g,h) = \frac{1}{2\pi} \int_{0}^{2\pi} f(t)\overline{g(-2t)h(t)} dt.
	\end{equation}
\end{lem}
\begin{proof}
	This lemma follows from a simple argument concerning comparison of the Fourier coefficients and the Parseval's theorem.
\end{proof}
\begin{proof}[Proof of Theorem~\ref{mrquant}]
	We use Lemma \ref{representation} and proceed similarly as in the proof of Theorem \ref{mainres2}, replacing the Gaussian kernel with the Fej\'er kernel.
\end{proof}

\begin{proof}[Proof of Theorem~\ref{mainres}]
	 Since in the set from the definition of $t_{n}$ (see Definition~\ref{tdef}) there are $4n+1$ pairs which correspond to trivial 3AP's, it suffices to show that for some $n$ we have
	 \begin{equation}
	 	t_{n}(\spec(\mu)) > 4n+1
	 \end{equation}
 which is implied by
	\begin{equation}
		\Lam(f_{n}, f_{n}, f_{n}) > 4n+1.
	\end{equation}
	This, in turns, follows from Theorem \ref{mrquant}. 
\end{proof}

\begin{proof}[Proof of Corollary \ref{wniosek}]
	
	This, as we have already mentioned, follows from the fact that (see Proposition 10.2. in \cite{Falconer}, and \cite{F06})
	\begin{equation}
		\dH(\mu) = \sup\{s: \underline{D}\mu(x) \geq s \ \text{for} \ \mu\text{-almost every} \ x\},
	\end{equation}
	and Theorem~\ref{mrquant} which gives that for all $x \in \T$ we have $\underline{D} \mu(x) \geq \frac{2-\beta}{3}$.
	
\end{proof}

\section{Proof of Theorem~\ref{comb2}}\label{chap_comb}

In this section we apply Theorem \ref{mainres2} to obtain new dimension bounds for vector-valued measures subordinated to bundles. Our criterion takes into account incidence properties of level sets of $\phi$, i.e.
\begin{equation*}
	\phi^{-1}(v) = \{\xi\in S^{n-1}: v \in \phi(\xi)\} \quad \text{for} \ v\in E.
\end{equation*}

As we have already mentioned, we will prove it in a slightly more general form than the one announced in the Introduction.

\begin{thm}\label{combinat}
Let $\phi: S^{n-1} \to \G_{\C}(d, E)$ be a smooth bundle and let $k\geq 1$ be a fixed integer. Suppose that for each $v \in E$ there exists orthogonal maps $B_{v}^{(1)}, \dots, B_{v}^{(k)} \in O(n)$ such that
\begin{equation*}
	\phi^{-1}(v) \cap B_{v}^{(1)}(\phi^{-1}(v))  \cap \dots  B_{v}^{(k)}(\phi^{-1}(v))= \emptyset.
\end{equation*}
Then
\begin{equation*}
	\dH(\mu) \geq \frac{n}{k+1}
\end{equation*}
for each $\mu \in \M_{\phi}(\R^{n}, E)$.
\end{thm}
This theorem obviously implies Theorem \ref{comb2}.
We derive this result by a modification of the strategy from \cite{dima}. Its main points are the following three steps:
\begin{enumerate}
	\item For any $\mu \in M_{\phi}(\R^{n}, E)$ and any $\epsilon > 0$ the author constructs a scalar, positive Radon measure $\nu$ which is a tempered distribution and such that:
	\begin{equation}\label{step1}
		\underline{D} \nu(0) \leq \dH(\mu) + \epsilon
	\end{equation}
	(Theorem 1.2. in \cite{dima}) and the spectrum (in the sense of distributions) of $\nu$ is contained inside a set 
	\begin{equation*}
		\{\lambda \xi \colon \lambda \geq 0,  \ \xi \in \phi^{-1}(v)\} \cup \{0\}
	\end{equation*}
	for some $v \in E$ (see formula (1.15) in \cite{dima}). 
	
	\item The measure $\nu$ is then replaced with a finite, non-negative measure $\theta$, $d\theta = f d\nu$, where $f \in \Schw(\R^{n})$. This measure, for a suitable $\phi$ still satisfies \eqref{step1} and its spectrum lies inside a $\delta$-neighbourhood of
	\begin{equation*}
		\{\lambda \xi \colon \lambda \geq 0,  \ \xi \in \phi^{-1}(v)\} \cup \{0\},
	\end{equation*}
	where $\delta = \text{diam}(\spec(\phi))$.
	
	\item Harnack-type inequality is used to give a lower bound of $\underline{D} \theta(0)$ in terms of parameters depending on $\phi$ (the maximal empty $k$-wave cone).
\end{enumerate}

Next we show that one can replace the use of the Harnack inequalities by an argument involving bounding a measure of suitable $\B$-configurations.

\begin{proof}[Proof of Theorem  \ref{combinat} assuming points (1) and (2) above]  Let us denote 
	\begin{equation*}
		F = \phi^{-1}(v)
	\end{equation*}
	and let $v$ be is as in (1) and (2). 
	
	Take orthogonal maps $B_{v}^{(1)}, B_{v}^{(2)}, \dots, B_{v}^{(k)} \in O(n)$ such that 
	\begin{equation*}
		F \cap B_{v}^{(1)}(F) \cap \dots \cap B_{v}^{(k)}(F) = \emptyset.
	\end{equation*} 
	Now let $\theta$ be a measure given by $(2)$ such that $\delta$ is so small that
	\begin{equation*}
		\delta (F) \cap B_{v}^{(1)} \delta (F) \cap \dots \cap B_{v}^{(k)} \delta (F) = \emptyset,
	\end{equation*}
	where $\delta(F)$ stands for the $\delta$-neighbourhood of $F$. It suffices to bound $\underline{D}  \theta(0)$. By the above, the set
	\begin{equation*}
		\{\xi \in \spec(\theta): \xi \ \text{generates a} ~ \B\text{-configuration}\},
	\end{equation*}
	for $\B$ given by  $\{B_{v}^{(1)}, \dots, B_{v}^{(k)}\}$, must lie inside the unit ball, and in particular it has finite Lebesgue measure. By Theorem \ref{mainres2}, we get $\underline{D}  \mu(0) \geq \frac{n}{k+1}$ which proves the result.
\end{proof}

\begin{remark}\label{Schwartz}
	The method from \cite{dima} applies also to measures from $\mathcal{D}'(\R^{n})$ satisfying PDE constraints (see Remark 1.2. therein). This means that one can apply Theorem \ref{combinat} (for bundles given by differential operators) also for this class.
\end{remark}

\section{Proof of Theorem \ref{sharper}}
We will construct a suitable operator on $\R^{3}$. Firstly, let us reformulate our goal in the language of bundles (see the Introduction). We need to construct $\phi : S^{2} \to \G_{\C}(1, \C^{3})$ such that:
\begin{enumerate}
	\item[(A)] $\phi(x,y,z) = \spann \{(P_{1}(x,y,z), P_{2}(x,y,z), P_{3}(x,y,z))\}$ and $P_{i}$ are real homogeneous polynomials of the same degree and 
	\begin{equation*}
		P(x,y,z) := (P_{1}(x,y,z), P_{2}(x,y,z), P_{3}(x,y,z)) \neq \vec{0} \quad \text{for}~ (x,y,z) \neq \vec{0}. 
	\end{equation*}
	\item[(B)] There exists $w \in \R^{3} \setminus \{0\}$ such that 
	\begin{equation*}
		\phi^{-1}(w) \cap V \neq \emptyset \quad \text{for any} \ V \in G(2, \R^{3}).
	\end{equation*}
	\item[(C)] For each $v \in \R^{3} \setminus \{0\}$ there exists an orthogonal map $B_{v} \in O(3)$ such that
	\begin{equation*}
		B_{v}(\phi^{-1}(v)) \cap \phi^{-1}(v) = \emptyset.
	\end{equation*}
\end{enumerate}
Firstly, let us observe that a measure subordinated to a bundle given by the polynomial $P$ is $\Ab$-free for some differential operator $\Ab$. Indeed, if we choose $\Ab$ to be an operator with (principal) symbol

\begin{align*}
\A(x,y,z) &= \begin{bmatrix}
			P(x,y,z) \times (1,0,0) \\
			(P(x,y,z) \times (1,0,0)) \times (1,0,0) \\
			P(x,y,z) \times (0,1,0) \\
			(P(x,y,z) \times (0,1,0)) \times (0,1,0)
\end{bmatrix},
\end{align*}
then $\ker \A(x,y,z) = \spann\{(P_{1}(x,y,z), P_{2}(x,y,z), P_{3}(x,y,z))\}$. (Here, the symbol $\times$ denotes the cross-product of vectors.) This identity holds, because either first two or second two rows of the matrix are non-zero, linearly independent vectors, and $P(x,y,z)$ is orthogonal to all of them. Moreover, (B) implies that $\Lambda_{2}(\Ab) \neq \{0\}$, while (C) is needed to apply Theorem \ref{combinat} and Remark \ref{Schwartz} which gives the desired dimensional bound.

To fullfil (B) and (C), we will construct our polynomial $P = (P_{1}, P_{2}, P_{3})$ so that it will have a level set similar to the set $\Gamma$ from the lemma below, and so that it will be locally injective outside a small $\delta$-neighbourhood of $\Gamma$.
 
\begin{lem}\label{curve}
There exists a centrally symmetric $1$-dimensional set $\Gamma \subset S^{2}$ such that
\begin{equation*}
	V \cap \Gamma \neq \emptyset \quad \text{for each} \ V \in \G(2, \R^{3})
\end{equation*}
and there exists a rotation $R \in O(3)$ such that
\begin{equation*}
	R (\Gamma) \cap \Gamma = \emptyset.
\end{equation*}
\end{lem}
\begin{proof}
One can choose $\Gamma$ whose construction is suggested by the descriptions (and figures) below.

\begin{figure}[H]
	\centering
	\includegraphics{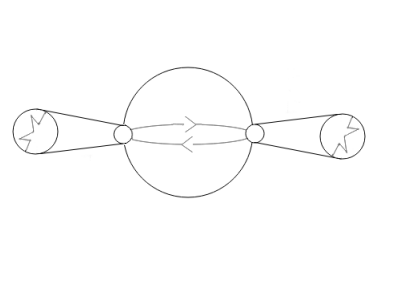}
	\caption{A part of construction: surgeries near the points $(1,0,0)$ and $(-1,0,0)$.}
\end{figure}

The set $\Gamma$ is simply derived from the equator $S^{2} \cap \spann \{(1,0,0), (0,1,0)\}$ by performing four surgeries. Two of them leave 'holes' near the points $(0,1,0)$ and $(0,-1,0)$ so that $\Gamma$ hits any $V \in \G(2, \R^{3})$. Remaining two modify this circle near the points $(1,0,0)$ and $(-1,0,0)$ so that after application of $R$, being a composition of a rotation by $90^{\circ}$ around $(1,0,0)$ and a rotation by $90^{\circ}$ around $(0,0,1)$, the modified part lies inside the 'hole', c.f. the figure below.
\begin{figure}[H]
	\centering
	\includegraphics{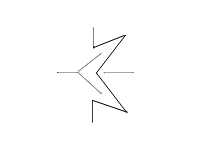}
	\caption{The result of the application of $R$ ($\Gamma$ - grey, $R(\Gamma)$ - black).}
\end{figure}

\end{proof}

For our applications we need $P_{1}, P_{2}, P_{3}$ to be homogeneous polynomials of the same degree.  In further steps, it will be more convenient for us to neglect the assumption about homogenity and to represent $P_{i}$ as arbitrary even polynomials. This can be done in view of the lemma below.

\begin{lem}
Let $F\in\R[x,y,z]$ be an even polynomial in three variables. Then $F_{|S^{2}} = \tilde{F}_{|S^{2}}$ for some homogeneous polynomial $\tilde{F}$.
\end{lem}
\begin{proof}
We have
\begin{equation*}
F(x,y,z) = \frac{F(x,y,z)+F(-x,-y,-z)}{2} = \sum^{M}_{j=0} F_{2j}(x,y,z) 
\end{equation*}
for some homogeneous polynomials $F_{2j}(x,y,z) $ of degree $2j$. On the other hand,
\begin{equation*}
F(x,y,z)_{|S^{2}} = \sum^{M}_{j=0} F_{2j}(x,y,z) (x^{2}+y^{2}+z^{2})^{M-j}.
\end{equation*}
\end{proof}
Let $C=\{a_{j}\}_{j=1}^{N}$ be a finite set of points from the unit sphere $S^{2}$, which is centrally symmetric in the sense that if $a_{j} \in C$, then $-a_{j} \in C$. For this set and some fixed $r>0$ we consider a family of balls $\mathcal{B} = \{B_{j}\}_{j=1}^{N}$,
\begin{equation*}
	B_{j} = \{p \in \R^{3}: \lVert p - a_{j} \rVert \leq\ r\}
\end{equation*}
and define 
\begin{equation*}
	T_{j} = S^{2}\cap\partial B_{j}.
\end{equation*}
Let us introduce a sequence of polynomials $Q_{j}$:
\begin{equation*}
	Q_{j}(x,y,z) = [(x-a_{j}^{(1)})^{2} + (y-a_{j}^{(2)})^{2} + (z-a_{j}^{(3)})^{2} - r^{2}]^{2}, \quad a_{j} = (a_{j}^{(1)}, a_{j}^{(2)}, a_{j}^{(3)}),
\end{equation*}
and denote 
\begin{equation*}
	Q(x,y,z) = \prod^{N}_{j=1} Q_{j}(x,y,z).
\end{equation*}

The desired polynomials $P_{i}$ will be of the form
\begin{equation}\label{eq1}
	P_{1}(x,y,z) = Kx^{2}Q(x,y,z), 
\end{equation}
\begin{equation}\label{eq2}
	P_{2}(x,y,z) = Ky^{2}Q(x,y,z),
\end{equation}
\begin{equation}\label{eq3}
	P_{1}(x,y,z) = 1+ Kz^{2}Q(x,y,z),
\end{equation}
for some family of balls $\mathcal{B}$ and a large constant $K>0$ that will be chosen in the next steps.

Firstly, let us choose the constants $K$ and $r$.
\begin{lem}
	Suppose that $P=(P_{1},P_{2},P_{3})$ is as above and
	\begin{equation*}
		\psi(x,y,z) = \spann \{(x^{2},y^{2},z^{2})\},
	\end{equation*}
	\begin{equation*}
		\phi(x,y,z) = \spann \{P(x,y,z)\}.
	\end{equation*}
	Then, for any sufficiently small $\delta > 0$ there exists $r>0$ and $K>0$ such that
	\begin{equation*}
		\lVert \phi - \psi \rVert_{\G_{\C}(1,\C^{3})} \lesssim \delta
	\end{equation*}
outside a $\delta$-neighbourhood of ~$\bigcup T_{i}$.
\end{lem}
\begin{proof}
	Let us choose $r = \frac{\delta}{2}$ and let us restrict our polynomials to the unit sphere. Outside $\bigcup T_{i}$ we have
\begin{equation*}
		P(x,y,z) = K Q(x,y,z) \tilde{P}(x,y,z)
\end{equation*}
for 
\begin{equation*}
	\tilde{P}(x,y,z) = (x^{2}, y^{2}, z^{2} + \frac{1}{KQ(x,y,z)}).
\end{equation*}
Since the metric on $\G_{\C}(1,\C^{3})$ (see formula \eqref{grasmet}) is equivalent to the angle between vectors, it suffices to prove that
\begin{equation*}
	\bigg{\lVert} \frac{(x^{2},y^{2},z^{2})}{\lVert (x^{2},y^{2},z^{2}) \rVert} -   \frac{\tilde{P}(x,y,z)}{\lVert \tilde{P}(x,y,z) \rVert} \bigg{\rVert} \lesssim \delta.
\end{equation*}
One can check, that by the triangle inequality 
\begin{equation*}
	\bigg{\lVert} \frac{(x^{2},y^{2},z^{2})}{\lVert (x^{2},y^{2},z^{2}) \rVert} -   \frac{\tilde{P}(x,y,z)}{\lVert \tilde{P}(x,y,z) \rVert} \bigg{\rVert} \lesssim \frac{1}{K|Q(x,y,z)|}.
\end{equation*}
Since outside $\delta$-neighbourhood of ~$\bigcup T_{i}$ 
\begin{equation*}
	|Q(x,y,z)| > (\delta - r)^{4N} = \bigg{(}\frac{\delta}{2}\bigg{)}^{4N},
\end{equation*}
it amounts to take $K \geq \frac{1}{\delta^{4N+1}}$.
\end{proof}

\begin{proof}[Proof of Theorem \ref{sharper}]
The kernel of the desired operator will be a polynomial bundle given by $\phi = \spann (P_{1}, P_{2}, P_{3})$ described earlier in the beginning of this section (see \eqref{eq1}, \eqref{eq2} and \eqref{eq3}) which will satisfy conditions (A), (B) and (C). Let us apply the choice of parameters from the previous lemma governed by $\delta$ from this lemma. This $\delta$ will be a variable that we will fix at the end of the proof. 

The condition (A) is clearly satisfied.

Now let us chose a suitable family of balls $\mathcal{B}$. We see, that if we take $\mathcal{B}$ to be a centrally symmetric family of balls with centers at $\Gamma$ (the set from Lemma \ref{curve}), each of radius $r=\delta/2$ such that $\Gamma \subset \bigcup_{i} B(a_{i},\delta/2)$, then (B) is guaranteed by the construction of $\Gamma$. Indeed, we see that
 \begin{equation*}
 \phi^{-1}((0,0,1)) \supset \bigcup T_{i} \cup \{(0,0,1)\} \cup \{(0,0,-1)\},
 \end{equation*}

 Moreover, by the previous lemma, for each $v = (v_{1}, v_{2}, v_{3}) \in S^{2}$, $\phi^{-1}(v)$ is contained in the sum of $\delta$-neighbourhood of $\bigcup T_{i} \cup \{(0,0,1), (0,0,-1)\}$ and at most eight balls of radii $\delta$, centered at points $p_{1}, \dots, p_{8}$ of the form
\begin{equation*}
	\frac{1}{\sqrt{|v_{1}|+|v_{2}|+|v_{3}|}}(\pm \sqrt{|v_{1}|}, \pm \sqrt{|v_{2}|}, \pm \sqrt{|v_{3}|}).
\end{equation*}

 Thus it suffices to prove (C). Let us fix $v$ and define
 \begin{equation*}
 \Gamma_{\delta}(v) = \bigcup_{i=1}^{N} B(a_{i}, 4\delta) \cup \bigcup_{i=1}^{8}B(p_{i}, \delta) \cup B((0,0,1),\delta) \cup B((0,0,-1),\delta).
 \end{equation*}
It remains to prove that we can perturb $R$ from Lemma \ref{curve} a little bit so that
 \begin{equation*}
 	\Gamma_{\delta}(v) \cap \tilde{R}(\Gamma_{\delta}(v)) = \emptyset,
 \end{equation*}
for some $\tilde{R}$ close to $R$. Indeed, put
\begin{equation*}
	S_{1} = \bigcup_{i=1}^{N} B(a_{i}, 4\delta), \quad S_{2} =B((0,0,1),\delta) \cup B((0,0,-1),\delta) \cup \bigcup_{i=1}^{8} B(p_{i}, \delta).
\end{equation*}
For $e \in S^{2}$, $e \neq (0,1,0)$ let $R_{e}$ be a rotation around $(0,1,0) \times e$, sending $(0,1,0)$ to $e$. It is easy to see that
\begin{equation*}
\HH^{2}_{|S^{2}}(e \in S^{2}: S_{1} \cap R_{e}R(S_{1}) = \emptyset) \geq c_{0} > 0,
\end{equation*}
where $c_{0}$ depends on $\Gamma$ and supremum of $\delta$'s only, because $\Gamma$ can be chosen in a way that 
\begin{equation*}
	S_{1} \cap R_{e}R(S_{1}) = \emptyset
\end{equation*}
holds for $e$ belonging to a small neighbourhood of $(0,1,0)$.
Moreover,
\begin{equation}
\HH^{2}_{|S^{2}}(e \in S^{2}: S_{i} \cap R_{e}R(S_{j}) \neq \emptyset) \lesssim \delta
\end{equation}
for $(i,j) \neq (1,1)$. Indeed, the set
\begin{equation}
	\{e \in S^{2}: S_{i} \cap R_{e}R(S_{j}) \neq \emptyset\}
\end{equation}
is contained either in a $10\delta$-neighbourhood of some finite set of points (if $(i,j) = (2,2)$) or in a $10\delta$-neighbourhood of a one-dimensional set whose length depends on $\Gamma$ (if $(i,j) = (1,2)$ or $(i,j) = (2,1)$).
By taking sufficiently small $\delta$, we obtain (C) (we choose $B_{v} = R_{e}R$ for some $e$). Clearly, our choice of $\delta$ is independent of $v$.

\end{proof}
\section{Acknowledgements}
The author wishes to express his gratitude to D. Stolyarov for valuable and encouraging discussions about this paper. The author also wishes to thank M. Wojciechowski for all years of support and mathematical guidance.

The work is supported by Ministry of Science and Higher Education of the Russian Federation, agreement № 075–15–2019–1620

The work is partially supported by the National Science Centre, Poland, CEUS programme, project no. 2020/02/Y/ST1/00072.
\bibliography{additive} 
\bibliographystyle{alpha}

\end{document}